\documentclass {article}

\usepackage{graphicx}
\usepackage{hyperref}
\usepackage{amsmath, amssymb, amsfonts, amsthm, enumerate}
\date{}

\newtheorem{theorem}{\bf Theorem}[section]
\newtheorem{claim}[theorem]{\bf Claim}
\newtheorem{lemma}[theorem]{\bf Lemma}
\newtheorem{conjecture}[theorem]{\bf Conjecture}

\newcommand{\todo}[1]{}

\title{Random subgraphs of 
properly edge-coloured complete graphs and long rainbow cycles}

\author{Noga Alon\thanks{Sackler School of Mathematics 
and Blavatnik School of
Computer Science, Tel Aviv University, Tel Aviv 69978, Israel.
Email: {\tt nogaa@tau.ac.il}.  Research supported in part by a
USA-Israeli
BSF grant 2012/107, by an ISF grant 620/13 and
by the Israeli I-Core program.}
\and
Alexey Pokrovskiy\thanks{Department of Mathematics, ETH, 8092 Zurich, 
Switzerland. {\tt dr.alexey.pokrovskiy@gmail.com}. 
Research supported in part by SNSF grant 200021-149111.}
\and
Benny Sudakov
\thanks{Department of Mathematics, ETH, 8092 Zurich, Switzerland. 
{\tt benjamin.sudakov@math.ethz.ch}. 
Research supported in part by SNSF grant 200021-149111.}
}

\begin{document}
\maketitle

\begin{abstract}
A subgraph of an edge-coloured complete graph is called rainbow if all
its edges have different colours.  In 1980 Hahn conjectured that every
properly edge-coloured complete graph $K_n$ has a rainbow Hamiltonian
path.  Although this conjecture turned out to be false, it was widely
believed that such a colouring always contains a rainbow cycle of length
almost $n$.  In this paper, improving on several earlier results,  we
confirm this by proving that every properly edge-coloured $K_n$ has a
rainbow cycle of length $n-O(n^{3/4})$. One of the main ingredients of
our proof, which is of independent interest, shows that a random subgraph
of a properly edge-coloured $K_n$ formed by the edges of a random set of
colours has a similar edge distribution as a truly random graph with the
same edge density. In particular it has very good expansion properties.
\end{abstract}

\section{Introduction}
In this paper we study properly edge-coloured complete graphs, i.e.,
graphs in which edges which share a vertex have distinct colours. Properly
edge-coloured complete graphs are important objects because they
generalize $1$-factorizations of complete graphs. A $1$-factorization
of $K_{2n}$ is a proper edge-colouring of $K_{2n}$ with $2n-1$ colours,
or equivalently a decomposition of the edges of $K_{2n}$ into perfect
matchings. These factorizations were introduced by Kirkman more than 150
years ago and were extensively studied in the context of combinatorial
designs (see, e.g., \cite{MR, Wal} and the references therein.)

A \emph{rainbow} subgraph of a properly edge-coloured complete graph
is a  subgraph all of whose edges have different colours.  One reason
to study such subgraphs arises in Ramsey theory, more precisely in the
canonical version of Ramsey's theorem, proved by Erd\H{o}s and Rado. Here
the goal is to show that edge-colourings of $K_n$, in which each colour
appears only few times contain rainbow copies of certain graphs (see,
e.g., \cite{SV}, Introduction for more details). Another  motivation
comes from problems in design theory. For example a special case of
the Brualdi-Stein Conjecture about transversals in Latin squares is that
every $1$-factorization of $K_{2n}$ has a rainbow subgraph with $2n-1$ edges
and maximum degree $2$. A special kind of a graph with maximum degree $2$
is   a Hamiltonian path, that is,  a path which goes through every vertex of
$G$ exactly once. Since properly coloured complete graphs are believed
to contain large rainbow maximum degree $2$ subgraphs, it is natural
to ask whether they have rainbow Hamiltonian paths as well. This was
conjectured by Hahn~\cite{Hahn} in 1980. Specifically he
conjectured  that
for $n\geq 5$, in every colouring of $K_n$ with $\leq n/2$ edges of every colour there is a rainbow Hamiltonian path.

%\begin{conjecture}
%For $n\geq 5$, every properly coloured $K_n$ has 
%a rainbow Hamiltonian path.
%\end{conjecture}

It turns out that this conjecture is false and for $n=2^k$, Maamoun
and Meyniel~\cite{Maamoun_Meyniel} found $1$-factorizations of $K_{n}$
without a rainbow Hamilton path. Nevertheless, it is widely believed
(see e.g., \cite {Gyarfas_Ruszinko_Sarkozy_Schelp}) that the intuition
behind Hahn's Conjecture is correct and that various slight weakenings
of this conjecture should be true.
Moreover, they should hold not only for $1$-factorizations but in general
for proper edge-colourings.  Hahn and Thomassen ~\cite{Hahn_Thomassen}
suggested that every properly edge-coloured $K_n$, with $<n/2$ edges of
each colour, has a rainbow Hamiltonian path. Akbari, Etesami, Mahini
and 
Mahmoody~\cite{Akbari_Etesami_Mahini_Mahmoody}  conjectured that every
$1$-factorization of $K_n$ contains a Hamiltonian cycle which has at
least $n-2$ different colours on its edges. They also asked whether every
$1$-factorization has a rainbow cycle of length at least $n-2$. Andersen
\cite{Andersen} further conjectured that every properly edge-coloured $K_n$
has a long  rainbow path.
\begin{conjecture}[Andersen, \cite{Andersen}]
Every properly edge-coloured $K_n$ has a rainbow path of length $n-2$.
\end{conjecture}

There have been many positive results supporting the above conjectures.
By a trivial greedy argument every  properly coloured $K_n$ has a rainbow
path of length $\geq n/2-1$. Indeed if the maximum rainbow path $P$ in
such a complete graph has length less than $n/2-1$, then an endpoint of
$P$ must have an edge going to $V(K_n)\setminus V(P)$ in a colour which
is not present in $P$ (contradicting the maximality of $P$.)  Akbari,
Etesami, Mahini and  Mahmoody~\cite{Akbari_Etesami_Mahini_Mahmoody} showed
that every properly coloured $K_n$ has a rainbow cycle of length $\geq
n/2-1$.  Gy\'arf\'as and Mhalla~\cite{Gyarfas_Mhalla} showed that every
$1$-factorization of $K_n$ has a rainbow path of length $\geq (2n+1)/3$.
Gy\'arf\'as, Ruszink\'o, S\'ark\"ozy,
and Schelp~\cite{Gyarfas_Ruszinko_Sarkozy_Schelp} showed that
every properly coloured $K_n$ has a rainbow cycle of length $\geq
(4/7-o(1))n$. Gebauer and Mousset~\cite{Gebauer_Mousset}, and
independently Chen and Li~\cite{Chen_Li} showed that every properly
coloured $K_n$ has a rainbow path of length $\geq (3/4-o(1))n$. But
despite all these results Gy\'arf\'as and Mhalla~\cite{Gyarfas_Mhalla}
remarked that ``presently finding a rainbow path even with $n - o(n)$
vertices is out of reach.''

In this paper we improve on all the above mentioned results by showing that every properly edge-coloured $K_n$ has an almost spanning rainbow cycle.
\begin{theorem}\label{TheoremRainbowCycle}
For all sufficiently large $n$, every properly edge-coloured $K_n$ contains a rainbow cycle of length at least $n-24n^{3/4}$.
\end{theorem}

\noindent

This theorem gives an approximate version of Hahn's and Andersen's
conjectures, leaving as an open problem to pin down the correct order of
the error term (currently between $-1$ and $-O(n^{3/4})$). The constant
in front of $n^{3/4}$ can be further improved and we make no attempt to
optimize it.

The proof of our main theorem is based on the following result, which has
an independent interest. For a graph $G$ and two sets $A,B\subseteq V(G)$,
we use $e_G(A,B)$ to denote the number of edges of $G$ with one vertex
in $A$ and one vertex in $B$. We show that the subgraph of a properly
edge-coloured $K_n$ formed by the edges in a random set of colours has
a similar edge distribution as a truly random graph with the same edge
density. Here we assume that $n$ is sufficiently large and write $f \gg g$
if $f/g$ tends to infinity with $n$.

\begin{theorem}\label{TheoremRandom}
Given a proper edge-colouring of $K_n$, let $G$ be the subgraph obtained by
choosing every colour class randomly and independently with probability
$p\leq 1/2$. Then, with high probability, all vertices in $G$ have degree
$(1-o(1))np$ and for every two disjoint subsets $A, B$ with $|A|,|B|\gg
(\log n/p)^2$, $e_{G}(A,B) \geq (1-o(1))p |A||B|$.
\end{theorem}

\noindent
Our proof can be also used to show that this conclusion holds for 
not necessarily disjoint sets (where in this case edges with both
endpoints in $A \cap B$ are counted twice). 
% of different sizes as long as these sizes are $\gg (\log n/p)^2$. 
For slightly larger sets $A, B$ of size at least
$\tilde{\Omega}( \log^2 n/p^2)$ (where here $\tilde{\Omega}(x)$
denotes, as usual, $x (\log x)^{O(1)}$) 
we can also obtain a
corresponding upper bound, showing that $e_{G}(A,B) \leq (1+o(1))p x^2$
(see remark in the next section).  Note that the edges in the random
subgraph $G$ are highly correlated. Specifically, the edges of the same
colour either all appear or all do not appear in $G$. Yet we show
that the edge distribution between all sufficiently large sets are not
affected much by this dependence.

\subsection*{Notation}
%A path of length $\ell$ in a graph is a sequence of $\ell+1$ distinct 
%vertices $p_0, p_1, \dots, p_{\ell}$ such that $p_ip_{i+1}$ is an edge. 
For two disjoint sets of vertices $A$ and $B$, we use $E(A,B)$ to denote
the set of edges between $A$ and $B$.  A \emph{path forest} $\mathcal
P=\{P_1, \dots, P_k\}$ is a collection of vertex-disjoint paths in
a graph.
%We use $\bigcup \mathcal P$ to denote the subgraph $P_1\cup \dots\cup P_k$.
For a path forest $\mathcal P$, let $V(\mathcal P)=V(P_1)\cup\dots\cup
V(P_k)$ denote the vertices of the path forest, and let $E(\mathcal
P)=E(P_1)\cup\dots\cup E(P_k)$ denote the edges of the path forest.  We'll
use additive notation for concatenating paths i.e. if $P=p_1p_2\dots p_i$
and $Q=q_1q_2\dots q_j$ are two vertex-disjoint paths and $p_iq_1$ is an
edge, then we let $P+Q$ denote the path $p_1p_2\dots p_i q_1q_2\dots q_j$.
For a graph $G$ and a vertex $v$, $d_G(v)$ denotes the number of edges in
$G$ containing $v$.  The minimum and maximum degrees of $G$ are denoted
by $\delta(G)$ and $\Delta(G)$ respectively.  For the sake of clarity,
we omit floor and ceiling signs where they are not important.

\section{Random subgraphs of properly coloured complete graphs}

The goal of this section is to prove that the edges from a random
collection of colours in a properly edge-coloured complete graph have
distribution similar to truly random graph of the same density. Our main
result here will be Theorem~\ref{TheoremRandom}. Throughout this section
we assume that the number of vertices $n$ is sufficiently large and all
error terms $o(1)$ tend to zero when $n$ tends to infinity. We say that
some probability event holds {\em almost surely} if its probability
is $1-o(1)$. Our approach here is similar to the one in \cite{AO},
section 3.3,
see also \cite{AKKR}, section 5.2, but requires several additional
ideas.
First we need to recall the well known Chernoff bound,
see e.g., \cite{AS}.

\begin{lemma}\label{Chernoff}
%Let $X=X_1+\dots X_n$ where $X_1, \dots, X_n$ are independent 
%Bernoulli random variables with probability $p$. 
Let $X$ be the binomial random variable with parameters $(n,p)$. 
Then for $\varepsilon\in (0,1)$ we have
$$\mathbb P\big(|X-pn|> \varepsilon pn\big)\leq 
2e^{-\frac{pn\varepsilon^2}{3}}.$$
\end{lemma}
Given a proper edge-colouring of $K_n$, we call a pair of disjoint subsets
$A, B$ \emph{nearly-rainbow} if the number of colours of edges between $A$
and $B$ is at least $(1-o(1))|A||B|$. The following lemma shows that we
can easily control the number of random edges inside nearly-rainbow pairs.
\begin{lemma}\label{LemmaNearRainbowExpansion}
Given a proper edge-colouring of $K_n$ , let $G$ be a subgraph of $K_n$
obtained by choosing every colour class with probability $p$.  Then,
almost surely, all nearly-rainbow pairs $A, B$ with  $|A|=|B|=y \gg \log
n/p$ satisfy $e_{G}(A,B) \geq (1-o(1))p y^2$
\end{lemma}
\begin{proof}
Since $y\gg \log n/p$, we can choose $\varepsilon=o(1)$ so that every
nearly-rainbow pair $A, B$ has at least $(1-\varepsilon/2)|A||B|$ colours,
and $\varepsilon^2y\geq 30\log n/p$.  Let $(A,B)$ be a nearly-rainbow
pair with $|A|=|B|= y$. Then the number of different colours in  $G$
between $A$ and $B$ is binomially distributed with parameters $(m,
p)$, where $m\geq (1-\varepsilon/2)y^2$ is the number of colours in
between $A$ and $B$ in $K_n$.  Since $e_G(A,B)$ is always at least
the number of colours between $A$ and $B$,  Lemma \ref{Chernoff}
implies $$\mathbb{P}\big(e_{G}(A,B) \leq (1-\varepsilon)p y^2\big) \leq
e^{-\varepsilon^2py^2/13}.$$ Since $\varepsilon^{2}y\geq 30\log n/p$,
the result follows by taking a union bound over all ${n \choose y}^2$
pairs of sets $A, B$ of size $y$ and all $y \leq n/2$.
\end{proof}
Our next lemma shows that we can partition any pair of sets $A, B$
into few parts such that almost all pairs of parts are nearly-rainbow.
\begin{lemma}\label{LemmaManyNearRainbowPairs}
Let $A, B$ be two subsets of size $x$ of properly edge-coloured $K_n$
and let $y$ satisfy $x\gg y^2$. Then there are partitions of $A$ and $B$
into sets $\{A_i\}$ and $\{B_j\}$ of size $y$ such that all but an $o(1)$
fraction of pairs $A_i, B_j$ are nearly-rainbow.
\end{lemma}

\begin{proof}

Since $x\gg y^2$ for $\varepsilon=o(1)$ we can assume that  $x\geq
\varepsilon^{-2}y^2$ and $x$ is divisible by $y$. To prove the lemma, we
will show that there are partitions of $A$ and $B$ into sets $\{A_i\}$
and $\{B_j\}$ of size $y$ such that all but an $\varepsilon$ fraction
of pairs $A_i, B_j$ are nearly-rainbow.

Consider pair of random subsets $S \subset A$ and $T \subset B$ of
size $y$ chosen uniformly at random from all such subsets. For every
colour $c$, let $E_c$ be the set of edges of colour $c$ between $A$ and
$B$. Given any two vertices $a,a'\in A$ notice that $\mathbb{P}(a\in
S)=y/x$ and $\mathbb{P}(a,a'\in S)=\frac{y(y-1)}{x(x-1)}$. The
same estimates hold for vertices  in $T\subseteq B$. This implies
that for two disjoint edges $ab$ and $a'b'$ between $A$ and $B$ have
$\mathbb{P}(ab\in E(S,T))=y^2/x^2$ and $\mathbb{P}(ab, a'b'\in E(S,T))=
\frac{y^2(y-1)^2}{x^2(x-1)^2}$. Also note that $|E_c|\leq x$, since the
edge-colouring on $K_n$ is proper.  Thus, by the inclusion-exclusion
formula we can bound the probability that a colour $c$ is present in
$E(S,T)$ as follows
\begin{eqnarray*}
\mathbb{P}(\text{$c$ present in $E(S,T)$}) &\geq& \sum_{e\in E_c}
\mathbb{P}(e\in E(S,T))-\sum_{e,f\in E_c}\mathbb{P}(e,f\in E(S,T))
\geq \frac{y^2}{x^2}|E_c|-\frac{y^2(y-1)^2}{x^2(x-1)^2}\binom{|E_c|}2\\
&=& \frac{y^2}{x^2}|E_c|\Big(1-(y-1)^2/(x-1)\Big) \geq 
\frac{y^2}{x^2}|E_c|\Big(1-y^2/x\Big) \geq 
\frac{y^2}{x^2}|E_c|(1-\varepsilon^2).
\end{eqnarray*}
Let $Z$ be the number of colours in $E(S,T)$. Note that 
$\sum_c |E_c|=x^2$. Hence, by linearity of expectation, 
$\mathbb{E}(Z)\geq\sum_c(1-\varepsilon^2)\frac{y^2}{x^2}|E_c|
=(1-\varepsilon^2)y^2.$ 
Since $Z\leq e(S,T)=y^2$ we have that $y^2-Z$ is non-negative with
$\mathbb{E}(y^2-Z)\leq \varepsilon^2y^2$. Therefore, by Markov's
inequality we have $\mathbb{P}(y^2-Z\geq \varepsilon y^2)\leq
\varepsilon$. This implies that with probability at least $1-\varepsilon$
a pair $S, T$ is nearly-rainbow.

Let $\{A_i\}$ and $\{B_j\}$ be random partitions of $A$ and $B$ into
sets of size $y$. By the above discussion the expected fraction of
pairs which are not nearly-rainbow $A_i, B_j$ is at most $\varepsilon$.
Therefore there exists some partition satisfying the assertion of
the lemma.  \end{proof}

Combining the above two lemmas we can now complete the proof of 
Theorem~\ref{TheoremRandom}.

\begin{proof}[Proof of Theorem~\ref{TheoremRandom}]

Given a proper edge-colouring of $K_n$ , let $G$ be a subgraph of $K_n$
obtained by choosing every color class with probability $p$. Since all the
edges incident to some vertex have distinct colours the degrees of $G$ are
binomially distributed with parameters$(n-1, p)$. Moreover by the condition
$x \gg (\log n/p)^2$ we have that $pn \gg \log n$. Therefore for every
vertex $v$, by the Chernoff bound the probability that $|d_G(v)-np|\geq
\epsilon np$ is at most $e^{\frac{-pn\varepsilon^2}{4}}$.  By the union
bound, all the degrees are almost surely $(1-o(1))np$.

Fix some $x\gg  (\log n/p)^2$. Notice that for any pair of
disjoint sets $A, B$ with $|A|, |B|\gg x$, $E_{K_n}(A,B)$ contains
$(1-o(1))|A||B|/x^2$ edge-disjoint pairs $A_i\subseteq A, B_j\subseteq B$
with $|A_i|=|B_j|=x$. Using this, it is sufficient to prove the theorem
just for pairs of sets $A, B$ with $|A|, |B|=x$

Let $y$ be some integer satisfying $y \gg \log n/p$ and $y^2
\ll x$, which exists since $x \gg (\log n/p)^2$.  Then, by Lemma
\ref{LemmaNearRainbowExpansion}, we have that for every nearly-rainbow
pair $S, T$ of sets of size $y$ there are at least $(1-o(1))py^2$ edges of
$G$ between $S$ and $T$. Let $A$ and $B$ be two arbitrary subsets of $G$
of size $x$. By Lemma~\ref{LemmaManyNearRainbowPairs}, there are partitions
$\{A_i\}, \{B_j\}$  of $A$ and $B$ into subsets of size $y$ such that all
but $o(1)$ fraction of the pairs $A_i, B_j$ are nearly-rainbow in $K_n$.
Then, almost surely,
$$
e_G(A,B)\geq \sum_{\text{nearly-regular }A_i, B_j} 
e_G(A_i,B_j) \geq (1-o(1)) \frac{x^2}{y^2} \cdot (1-o(1))py^2
\geq (1-o(1))px^2,
$$
completing the proof. 
%\hfill $\Box$
\end{proof}
\vspace{0.2cm}

\noindent
{\bf Remark.}\, If $x$ is a bit larger, being at least, say,
$\frac{\log^2 n}{p^2} (\log ( \frac{\log n}{p}))^{O(1)}$ 
then one can modify our proof to
also bound $e_G(A,B)$  from above by $(1+o(1))px^2$. Indeed in this
case we can split each of the two sets $A$ and $B$ of size $x$ 
to disjoint subsets of size,
say $y=\sqrt{x}/100$ and show that with positive probability no pair of 
these subsets spans more than $\log x$ edges of the same color. 
When we pick each color randomly and independently with probability $p$ 
the expected number of edges between any two such subsets with the
colors picked is exactly $py^2$. The number of edges between any such sets is also $c$-Lipschitz with $c=\log x$ since by the above discussion, changing the decision about one colour can change the quantity by at most $\log x$.
Therefore we can replace the Chernoff
bound by Azuma's Inequality (see~\cite{McD}) to conclude that with high probability 
the number will be $(1+o(1))py^2$ for each such pair. We omit the
details as the upper bound will not be needed in this paper.
\iffalse
{\bf Remark.}\, If $x \gg \log^2 n/p^4$ then one can use our proof to
also bound $e_G(A,B)$  from above by $(1-o(1))px^2$. Indeed in this
case we can first change the definition of nearly-rainbow pair to have
at least $1-o(p)$ fraction of the edges with distinct colours  and
then  adjust Lemma \ref{LemmaManyNearRainbowPairs} to show that for
$y^2 \ll p^2x$ there exist partitions with at most a $o(p)$ fraction
of non nearly-rainbow pairs. Then in the above proof even if all non
nearly-rainbow pairs are complete and all the edges with non-distinct
colours in nearly-rainbow pairs are present it can only contribute at
most $o(px^2)$ edges.
\fi
\section{Rainbow path forest}
The following lemma is the second main ingredient which we will need to
prove Theorem~\ref{TheoremRainbowCycle}. It says that every properly
coloured graph with very high minimum degree has a nearly-spanning
rainbow path forest. The proof is a version of a technique of
Andersen~\cite{Andersen} who proved the same result for complete graphs.
\begin{lemma}\label{LemmaPathForest}
For all $\gamma, \delta, n$ with $\delta\geq \gamma$ and  $3\gamma\delta
-\gamma^2/2>n^{-1}$ the following holds.  Let $G$ be a properly coloured
graph with $|G|= n$ and $\delta(G)\geq(1-\delta)n$. Then $G$ contains a
rainbow path forest with $\leq \gamma n$ paths  and $|E(\mathcal P)|\geq
(1-4\delta)n$.
\end{lemma}
\begin{proof}
Let  $\mathcal P=\{P_1, \dots, P_{\gamma n}\}$ be a rainbow path
forest with $\leq \gamma n$ paths and $|E(\mathcal P)|$ as large as
possible. Suppose for the sake of contradiction that $|E(\mathcal P)|<
(1-4\delta)n$.  We claim that without loss of generality we may suppose
that all the paths $P_1, \dots, P_{\gamma n}$ are nonempty. Indeed
notice that we have $|V(\mathcal P)|\leq |E(\mathcal P)|+\gamma n<
(1-4\delta)n+ \gamma n\leq n-\gamma n$. Therefore if any of the paths
in $\mathcal P$ are empty, then we can replace them by single-vertex
paths outside $V(\mathcal P)$ to get a new path forest with the same
number of edges as $\mathcal P$.  For each $i$, let the path $P_i$ have
vertex sequence $v_{i,1}, v_{i,2}, \dots, v_{i,|P_i|}$. For a vertex
$v_{i,j}$ for $j>1$, let $e(v_{i,j})$ denote the edge $v_{i,j}v_{i,j-1}$
going from $v_{i,j}$ to its predecessor on $P_j$, and let $c(v_{i,j})$
denote the colour of $e(v_{i,j})$.

We define sets of colours $C_0, C_1, \dots, C_{\gamma n}$
recursively as follows. Let $C_0$ be the set of colours not on paths
in $\mathcal P$.  For $i=1, \dots, \gamma n$, let $$C_i=\left\{c(x):
x\in N_{C_{i-1}}(v_{i,1})\cap  V(\mathcal P)\setminus \{v_{1,1}, \dots,
v_{\gamma n,1}\}\right\}\cup C_{i-1}.$$ Notice that for any colour $c\in
C_i\setminus C_{i-1}$, there is an edge from $v_{i,1}$ to a vertex
$x\in V(\bigcup \mathcal P)$ with $c(x)=c.$

\begin{claim}\label{ClaimNCi}
$N_{C_{i-1}}(v_{i,1})\subseteq V(\mathcal P)\setminus 
\{v_{i+1,1}, \dots, v_{\gamma n,1}\}$ for $i=1\dots, \gamma n$.
\end{claim}
\begin{proof}
First we'll deal with the case when for $j>i$ there is an 
edge $v_{i,1}v_{j,1}$ by something in $C_{i-1}$. 
Define integers, $s$, $i_0, \dots, i_s$, colours $c_1, \dots, c_s$,
and vertices $x_0, \dots, x_{s-1}$ as follows.
\begin{enumerate}[(1)]
\item Let $i_0=i$ and $x_0=v_{j,1}$.
\item We will maintain that if $i_t\geq 1$, then the colour of 
$v_{i_{t},1}x_{t}$ is in $C_{(i_{t})-1}$. 
Notice that this does hold for $i_0$ and $x_0$.
\item For $t\geq 1$, let $c_t$ be the colour of $v_{i_{t-1},1}x_{t-1}$. 
By (2), we have $c_t\in C_{(i_{t-1})-1}$.
\item For $t\geq 1$, let $i_t$ be the smallest number for 
which $c_t\in C_{i_t}$. 
Notice that this ensures $c_t\in C_{i_t}\setminus C_{(i_t)-1}$
\item For $t\geq 1$, if $i_t>0$ then let $x_t$ be the vertex of
$V(\mathcal P)$ with $c(x_t)=c_t$. Such a vertex must exist since
from (4) we have $c_t\in C_{i_t}\setminus C_0$. Notice that by the
definition of $C_{i_t}$ and $c_t\in C_{i_t}\setminus C_{(i_t)-1}$,
the edge $v_{i_t,1}x_t$ must be present and coloured by something in
$C_{(i_t)-1}$ as required by (2).
\item We stop at the first number $s$ for which $i_s=0$.
\end{enumerate}
See Figure~\ref{FigurePathForest} for a concrete 
example of these integers, colours, and vertices being chosen.
Notice that from the choice of $c_t$ 
and $i_t$ in (3) and (4) we have $i_0>i_1>\dots>i_s$.
We also have $x_t\neq x_{t'}$ for $t\neq t'$. To see this notice that
from (4) and (5)  we have $c(x_t)=c_t\in C_{i_t}\setminus C_{(i_t)-1}$
and $c(x_{t'})=c_{t'}\in C_{i_{t'}}\setminus C_{(i_{t'})-1}$. Since
the sets $C_0, C_1, \dots$ are nested, the only way $c(x_t)=c(x_{t'})$
could occur is if $i_t=i_{t'}$ (which would imply $t=t'$.)  The following
claim will let us find a larger rainbow path forest than $\mathcal P$.

\begin{figure}
  \centering
    \includegraphics[width=0.7\textwidth]{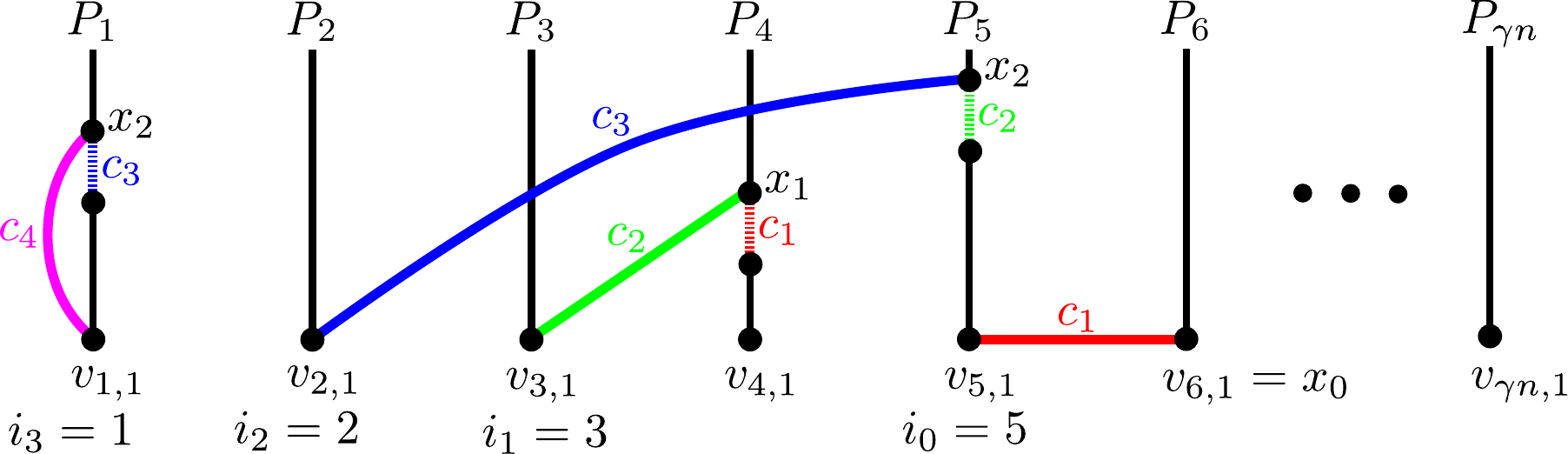}
  \caption{An example of the proof of Claim~\ref{ClaimNCi}. 
In this example $s=3$. The dashed coloured edges get deleted from the 
path forest and are replaced by the solid coloured ones. This gives a 
bigger path forest, contradicting the maximality of $\mathcal P$.}
\label{FigurePathForest}
\end{figure}
\begin{claim}
$\mathcal P'= P_1\cup\dots\cup P_{\gamma n}\cup 
\{(v_{i_0,1}x_0), (v_{i_1,1} x_1), (v_{i_2,1}x_2), 
\dots, (v_{i_{s-1},1}x_{s-1})\}\setminus\{e(x_1), e(x_2), 
\dots, e(x_{s-1})\}$ is a rainbow path forest.
\end{claim}
\begin{proof}
To see that $\mathcal P'$ is rainbow, notice that for $0\leq t<s-1$,
the edges $v_{i_t,1} x_t$ and $e(x_{t+1})$ both have the same colour,
namely $c_{t+1}$.  This shows that $\mathcal P'-v_{i_{s-1},1}x_{s-1}=
P_1\cup\dots\cup P_{\gamma n}\cup \{(v_{i_0,1}x_0),$ $(v_{i_1,1}
x_1),$ $(v_{i_2,1}x_2), \dots,$ $(v_{i_{s-2},1}x_{s-2})\}$ $\setminus$
$\{e(x_1),$ $e(x_2), \dots, e(x_{s-1})\}$ has exactly the same colours
that $\mathcal P$ had. By the definition of $s$, we have that the colour
$c_s$ of  $v_{i_{s-1},1}x_{s-1}$ is in $C_0$ and hence not in $\mathcal
P$,  proving that $\mathcal P'$ is rainbow.

To see that $\mathcal P'$ is a forest, notice that since $\mathcal P$
is a forest any cycle in $\mathcal P'$ must use an edge $v_{i_t,1} x_t$
for some $t$.  Let the vertex sequence of such a cycle be $v_{i_t,1}, x_t,
u_1, u_2, \dots, u_\ell, v_{i_t,1}$.  Since $x_t\in V(\bigcup \mathcal P)$
we have that $x_t=v_{k,j}$ for some $j$ and $k$.  Notice that since the
edge $e(x_t)=v_{k,j}v_{k,j-1}$ is absent in $\mathcal P'$, we have 
that $u_1=v_{k,j+1}$. Let $r$ be the smallest index for which $u_r\neq
v_{k,j+r}$. By the definition of $\mathcal P'$, we have that the edge
$u_{r-1}u_r$ must be of the form $v_{i_t',1}x_{t'}$ for some $t'\neq t$
with $u_{r-1}=x_{t'}$ and $u_{r}=v_{i_t',1}$. However, then the edge
$u_{r-2} u_{r-1}=e(x_{t'})$ would be absent, contradicting the fact
that $C$ is a cycle.

To see that $\mathcal P'$ is a path forest, notice that it has maximum
degree $2$---indeed the only vertices whose degrees increased are $x_0,
v_{i_0,1}, \dots, v_{i_{s-1},1}$. Their degrees increased from $1$ to
$2$ when going from $\mathcal P$ to $\mathcal P'$, which implies that
$\Delta(\mathcal P')\leq 2$ is maintained.
\end{proof}
Now  $\mathcal{P'}$ is a path forest with $\leq \gamma n$ paths and
one more edge than $\mathcal P$ had, contradicting the maximality of
$\mathcal P$.

The case when $v_{i,1}v$ is an edge for some $v\not\in V(\bigcup \mathcal
P)$ is identical using $x_0=v$.
\end{proof}
For $i=1, \dots, s$, let $m_i=|C_i|-|C_0|$. Let $C$ be the set of all 
the colours which occur in $G$. Notice that, by the definition 
of $C_0$, we have $|C|=|C_0|+e(\bigcup \mathcal P)$.
%Since every colour in $G$ occurs either in $C_0$ or in $\mathcal P$, we 
%have that m_i is exactly the number of colours from $\mathcal P$ 
%which are in $C_i$. 
Notice that for any vertex $v$ we have 
\begin{equation}\label{EqNCiSizeBound}
|N_{C_i}(v)|\geq |N(v)|-\big(|C|-|C_i|\big)\geq (1-\delta) 
n -(|C_0|+e(\bigcup \mathcal P))+(|C_0|+m_i)\geq 3\delta n+m_i.
\end{equation}
The first inequality comes from the fact that $G$ is properly coloured
and there are $|C|-|C_i|$ colours which are not in $C_i$. The second
inequality comes from $\delta(G)\geq (1-\delta)n$ and the definitions
of $m_i$ and $C_0$. The third inequality comes from $e(\bigcup \mathcal
P)\leq (1-4\delta)n$.

From the definition of $C_i$, we have $|C_i|\geq
|C_0|+|N_{C_{i-1}}(v_{i,1})\cap \{v_{k,j}: k\geq 2,
j=1, \dots, \gamma n\}|$. From Claim~\ref{ClaimNCi}, we have
$|N_{C_{i-1}}(v_{i,1})\cap \{v_{k,j}: k\geq 2, j=1, \dots, \gamma n\}|\geq
|N_{C_{i-1}}(v_{i,1})|-i$. Combining these with (\ref{EqNCiSizeBound})
we get $|C_i|\geq |C_0|+ 3\delta n + m_{i-1}-i$ which implies $m_{i}\geq
m_{i-1}+3\delta n-i$ always holds. Iterating this gives $m_i\geq 3i\delta
n-\binom i 2$. Setting $i=\gamma n$, gives $n\geq m_{\gamma n}\geq
3\gamma  \delta n^2- (\gamma n) ^2/2$, which contradicts $3\gamma\delta
-\gamma^2/2>n^{-1}$.
\end{proof}

\section{Long rainbow cycle}
In this section we prove Theorem~\ref{TheoremRainbowCycle}, using the
following strategy. First we apply Theorem~\ref{TheoremRandom} to find a
good expander $H$ in $K_n$ whose maximum degree is small and whose edges
use only few colours . Then we apply Lemma~\ref{LemmaPathForest} to find
a nearly spanning path forest with few paths, which shares no colours
with $H$. Then we use the expander $H$ to ``rotate'' the path forest
in order to successively extend one of the paths in it until we have a
nearly spanning rainbow path $P$. Finally we again use the expander $H$
to close $P$ into a rainbow cycle.

We start with the lemma which shows how to use an expander to 
enlarge one of the paths in a path forest. 
\begin{lemma}\label{LemmaRotation}
For $b, m, r>0$ with $2mr\leq b$, the following holds.
Let $\mathcal P=\{P_1, \dots, P_{r}\}$ be a rainbow path 
forest in a properly coloured graph  $G$.
Let $H$ be a subgraph of $G$ sharing no colours with $\mathcal P$ 
with $\delta(H)\geq 3b$ and $E_H(A,B)\geq b+1$ for any two sets 
of vertices $A$ and $B$ of size $b$. 
Then either  $|P_1|\geq |V(\mathcal P)| - 2b$ or there are 
two edges $e_1, e_2\in H$ and a rainbow path 
forest $\mathcal P' = \{P'_1, \dots, P'_{r}\}$  
such that  $E(\mathcal P')\subseteq E(\mathcal P')+e_1+e_2$ 
and $|P'_1|\geq |P_1|+m$.
\end{lemma}
\begin{proof}
Suppose that  $|P_1|< |V(\mathcal P)| - 2b$.  Let $P_1=v_1, v_2, \dots,
v_{k}$ and let $T$ be the union of vertices on those paths among $P_2,
\dots, P_{r}$ which have length  at least $2m$.  Notice that there are
at most $2mr$ vertices on paths $P_i$ of length $\leq 2m$. Since $|P_1|<
|V(\mathcal P)| - 2b$,  the set $T$ has size at least $2b-2mr\geq b$.

First suppose that there is an edge of $H$ from $v_1$ to a vertex $x
\in T$ on some path $P_i$ of length at least $2m$. We can partition
$P_i$ into two subpaths $P^+$ and $P^-$ such that $P^+$ starts with $x$
and has $|P^+|\geq |P_i|/2\geq m$.  Then we can take $e_1=e_2={v_1}x$,
$P_1=P_1+e_1+P^+$, $P_i=P^-$ and $P_j'=P_j$ for all other $j$ to obtain
paths satisfying the assertion of the lemma.

Next suppose that $|N_H({v_1})\cap P_1|\geq b.$ Let $S\subseteq
N_H({v_1})\cap P_1$  be a subset of size $b$. Let $S^+$ be the set
of predecessors on $P_1$ of vertices in $S$, i.e., $S^+=\{v_{i-1}:
v_i\in S\}$. Since $|S^+|, |T|\geq b$, there are at least $b+1$ edges
between $S^+$ and $T$ in $H$. In particular this means that there is
some $v_\ell\in S^+$ which has $|N_H(v_\ell)\cap T|\geq 2$. Since
$H$ is properly coloured, there is some $x\in N_H(v_\ell)\cap T$
such that $v_\ell x$ has a different colour then that 
of $v_1 v_{\ell+1}$. By
the definition of $T$, this $x$ belongs to path $P_i, i\geq 2$ with
$|P_i|\geq 2m$. Again we can partition $P_i$ into two subpaths $P^+$
and $P^-$ such that $P^+$ starts with $x$ and has $|P^+|\geq m$.  Then,
taking $e_1={v_1}v_{\ell+1}$ $e_2=v_\ell x$, $P_1=(v_k,v_{k-1},\dots,
v_{\ell+1} v_1,v_2, \dots, v_{\ell})+e_2+P^+$, $P_i'=P^-$ and $P_j'=P_j$
for all other $j$, we obtain paths satisfying the assertion of the lemma.

Finally we have that $|N_H({v_1})\cap P_1|< b$ and there are no edges
from $v_1$ to $T$. Note that there are also at most $2mr \leq b$
edges from $v_1$ to vertices on paths $P_i$ of length $\leq 2m$. Since
$|N_H(v_1)|\geq 3b$, there is a set $S\subseteq N_H(v_1)\setminus
V(\mathcal P)$ with $|S|=b$. Since $|S|, |T|\geq b$, there is an edge
$sx$ in $H$ from some $s\in S$ to some $x\in T$. Since $H$ is properly
coloured, $sx$ has a different colour then that of $v_1 s$.  
By the definition
of $T$, the vertex $x$ is on some path $P_i, i\geq 2$ of length at least
$2m$.  Partition $P_i$ into two subpaths $P^+$ and $P^-$ such that $P^+$
starts with $x$ and has $|P^+|\geq m$.  Let $e_1=v_1s$ $e_2=sx$. Then
$P_1'=P_1+e_1+e_2+P^+$, $P_i'=P^-$ and $P_j'=P_j$ for all other $j$
satisfy the assertion of the lemma, completing the proof
\end{proof}

Having finished all the necessary preparations we are now ready to
show that every properly edge-coloured $K_n$ has a nearly-spanning
rainbow cycle.

\begin{proof}[Proof of Theorem~\ref{TheoremRainbowCycle}]
Given a properly edge-coloured complete graph $K_n$, we first use
Theorem~\ref{TheoremRandom} to construct its subgraph $H$, satisfying
the conditions of Lemma  \ref{LemmaRotation}.  Let  $b = n^{3/4}$. Let $H$
be a subgraph obtained by choosing every colour class randomly and
independently with probability  $p= 4.5b/n$.  Since $p= 4.5 n^{-1/4}$
we have that  $b=n^{3/4} \gg n^{1/2} \log^2 n>(\log n/p)^2$. Therefore,
we can apply Theorem~\ref{TheoremRandom}  to get that almost surely
every vertex in $H$ has degree $4b \leq d_{H}(v) =(1-o(1))np\leq 5b-1$
and $e_{H}(A,B)\geq (1-o(1))pb^2>4.3{n}^{1/2}b$ for any two 
disjoint sets $A$ and $B$ of size $b$. Fix such an $H$.

Let $G=K_n\setminus H$ be the subgraph of $K_n$ consisting of all
edges whose colours are not in $E(H)$. We have $\delta(G)\geq
n-1-\Delta(H)\geq (1-5n^{-1/4})n$.  Applying Lemma~\ref{LemmaPathForest}
with $\delta=5n^{-1/4}$ and $\gamma=n^{-3/4}$ we get a rainbow path
forest $\mathcal P$ with $n^{1/4}$ paths and $|E(\mathcal P)|\geq
n-20n^{3/4}$. Moreover the colours of edges in $\cal P$ and $H$ are
disjoint.

Next, we repeatedly apply Lemma~\ref{LemmaRotation} $2n^{1/2}$ times
with $b=n^{3/4}$, $r=n^{1/4}$, and $m=0.5n^{1/2}$. At each iteration we
delete from $H$ all edges sharing a colour with $e_1$ or $e_2$ to get a
subgraph $H'$.  Notice that after $i$ iterations, $H$ has lost at most
$2i$ colours, and so $\delta(H')\geq \delta(H)-2i\geq 4b-2i > 3b$ and for
any $A, B\subseteq V(H)$ with $|A|, |B|\geq b$  we have $e_{H'}(A,B)\geq
4.3{n}^{1/2}b-2ib \geq 0.3{n}^{1/2}b>b+1$.  This shows that we indeed
can continue the process for $2n^{1/2}$ steps without violating the
conditions of Lemma~\ref{LemmaRotation}.  At each iteration we either
increase the length of $P_1$ by $m$, or we establish that $|P_1|\geq
|V(\mathcal P)|-2b$.  Since $2n^{1/2}m=n>n-2b$ we have that the second
option must occur at some point during the $2n^{1/2}$ iterations of
Lemma~\ref{LemmaRotation}.  This gives a rainbow path $P$ of length at
least $|V(\mathcal P)|-2b\geq n-22n^{3/4}$. As was mentioned above there
must still be at least $0.3n^{1/2}b$ edges of $H'$ left between any two
disjoint sets $A$, $B$ of size $b$.  Let $S$ be the set of first $b$
vertices and $T$ be the set of last $b$ vertices of the rainbow path
$P$. Then there is an edge of $H'$ between $S$ and $T$  whose colour is
not on $P$. Adding this edge we get a rainbow cycle of length at least
$|P|-2b \geq n-24n^{3/4}$, completing the proof.  \end{proof}

\section{Concluding remarks}

Versions of Theorem~\ref{TheoremRandom} can be proved in settings other
than properly coloured complete graphs. One particularly interesting
variation is to look at properly coloured balanced complete bipartite
graphs $K_{n,n}$.
\begin{theorem}\label{BipartiteTheorem}
Given a proper edge-colouring of $K_{n,n}$ with bipartition classes
$X$ and $Y$, let $G$ be a subgraph obtained by choosing every colour
class randomly and independently with probability $p$. Then, with high
probability, all vertices in $G$ have degree $(1-o(1))np$ and for every
two subsets $A\subseteq X, B\subseteq Y$ of size $x \gg (\log n/p)^2$,
$e_{G}(A,B) \geq (1-o(1))p x^2$.
\end{theorem}
The above theorem is proved by essentially the same argument as
Theorem~\ref{TheoremRandom}. Properly coloured balanced complete
bipartite graphs are interesting because they generalize Latin
squares. Indeed given any $n\times n$ Latin square, one can associate
a proper colouring of $K_{n,n}$ with $V(K_{n,n})=\{x_1, \dots, x_n,
y_1, \dots, y_n\}$ to it by placing a colour $i$ edge between $x_j$
and $y_k$ whenever the $(j,k)$th entry in the Latin square is $i$. Thus
Theorem~\ref{BipartiteTheorem} implies that every Latin square has  a
small set of symbols which exhibits a random-like behavior.

It would be interesting to find the correct value of the second order term
in Theorem~\ref{TheoremRainbowCycle}. So far, the best lower bound on
this is ``$-1$'' which comes from Maamoun and Meyniel's construction
in~\cite{Maamoun_Meyniel}. It is quite possible that their construction is
tight and ``$-1$'' should be the correct value. However this would likely
be very hard to prove since, at present, we do not even know how to get
a rainbow maximum degree $2$ subgraph of a properly coloured $K_n$ with
$n-o(\sqrt n)$ edges (a subgraph with $n-O(\sqrt n)$ edges can be obtained
by Lemma~\ref{LemmaPathForest}, or by a result from~\cite{Andersen}.)

Finally it would be interesting to know what is the smallest size of
the sets for which Theorem~\ref{TheoremRandom} holds. In particular, is
it true that for all $|A|, |B|\gg (\log n /p )^{1+\varepsilon}$ we have
$e_G(A,B)\geq (1-o(1))p|A||B|$, where $G$ is the graph in Theorem 1.3?
It can be shown that this is not the case if $p=1/2$ and $|A|=|B|=c \log
n \log \log n$ for an appropriate absolute constant  $c>0$ (see~\cite{Green}.)

\end{document}